\documentclass{siamltex}
\usepackage{amsfonts,amsmath,mathrsfs}
\usepackage{amssymb}
\usepackage{hyperref}
\usepackage{graphicx}
\usepackage{wrapfig}
\usepackage{subfig}
\usepackage{float}
\usepackage{bbm}
\usepackage[ruled,vlined,linesnumbered]{algorithm2e}
\usepackage{gensymb}
\usepackage{mathtools} 
\usepackage[capitalise,nameinlink]{cleveref}

\usepackage[textsize=tiny,textwidth=2.3cm]{todonotes}

\newcommand{\N}{\mathbb{N}}
\newcommand{\R}{\mathbb{R}}
\newcommand{\C}{\mathbb{C}}

\newcommand{\norm}[1]{\left\| #1 \right\|}

\def\Id{{\rm I}}
\DeclareMathOperator{\lev}{lev}
\newtheorem{remark}[theorem]{Remark}

\newcommand{\overbar}[1]{\mkern 1.0mu\overline{\mkern-1.0mu#1\mkern-1.0mu}\mkern 1.0mu}

\title{Electrodeless electrode model for electrical impedance tomography}

\author{J. Dard\'e\footnotemark[2]
\and N. Hyv\"onen\footnotemark[3]
\and T. Kuutela\footnotemark[3]
\and T. Valkonen\footnotemark[4]
}

\begin{document}
\maketitle

\renewcommand{\thefootnote}{\fnsymbol{footnote}}

\footnotetext[2]{Institut de Math\'ematiques de Toulouse, UMR 5219,  Universit\'e de Toulouse, CNRS,  UPS, F-31062 Toulouse Cedex 9, France (jeremi.darde@math.univ-toulouse.fr). The work of JD was supported by the Institut Français de Finlande, the Embassy of France in Finland, the French Ministry of Higher Education, Research and Innovation, the Finnish Society of Sciences and Letters and the Finnish Academy of Science and Letters (2019 Maupertuis Programme).}
\footnotetext[3]{Aalto University, Department of Mathematics and Systems Analysis, P.O. Box 11100, FI-00076 Aalto, Finland (nuutti.hyvonen@aalto.fi, topi.kuutela@aalto.fi). The work of NH and TK was supported by the Academy of Finland (decision 312124).}
\footnotetext[4]{Department of Mathematics and Statistics, University of Helsinki, Finland and ModeMat, Escuela Polit\'ecnica Nacional, Quito, Ecuador (tuomo.valkonen@iki.fi). The work of TV was supported by the Academy of Finland (decisions 314701 and 320022).}

\begin{abstract}
Electrical impedance tomography is an imaging modality for extracting information on the interior structure of a physical body from boundary measurements of current and voltage. This work studies a new robust way of modeling the contact electrodes used for driving current patterns into the examined object and measuring the resulting voltages. The idea is to not define the electrodes as strict geometric objects on the measurement boundary, but only to assume approximate knowledge about their whereabouts and let a boundary admittivity function determine the actual locations of the current inputs. Such an approach enables reconstructing the boundary admittivity, i.e.~the locations and strengths of the contacts, at the same time and with analogous methods as the interior admittivity. The functionality of the new model is verified by two-dimensional numerical experiments based on water tank data.
\end{abstract}

\renewcommand{\thefootnote}{\arabic{footnote}}



\pagestyle{myheadings}
\thispagestyle{plain}
\markboth{J. DARD\'E, N. HYV\"ONEN, T. KUUTELA, AND T. VALKONEN}{ELECTRODELESS ELECTRODE MODEL}

\section{Introduction}
\label{sec:introduction}

In {\em electrical impedance tomography} (EIT) contact electrodes are used for driving currents into an examined body and for measuring the resulting voltages, with the aim of reconstructing the interior admittivity of the body based on such boundary measurements. We refer to the review articles~\cite{Borcea02,Cheney99,Uhlmann09} for more information on EIT, including theoretical results on the unique identifiability of the interior admittivity assuming idealized boundary measurements. The goal of this work is to introduce a new way of modeling the electrodes in EIT based on a surface admittivity function on the boundary of the imaged object.

The most accurate model for electrode measurements of EIT is arguably the {\em complete electrode model} (CEM) that accounts for the shapes of the electrodes as well as for the thin resistive layers at the electrode-object interfaces~\cite{Cheng89,Somersalo92}. In its basic form the CEM assumes a constant contact on each electrode, but it also generalizes straightforwardly to the case of spatially varying contacts. Other useful electrode models include the {\em shunt model}~\cite{Cheng89} that can be considered as the limit of the CEM as the contacts become perfect~\cite{Darde16} and the {\em point electrode model} \cite{Hanke11b} that is the limit of the CEM for relative measurements as the electrodes converge to points on the object boundary. A common property of all these  models is that they treat the electrodes as subsets of the object boundary.

In this work, we do not assume to know the precise locations of the electrodes but only certain larger boundary sections, called {\em extended electrodes} (cf.~\cite{Hyvonen09}), that correspond to the available prior information on the whereabouts of the true electrodes. For each true electrode, there is exactly one extended electrode. The prior information on the positions of the true electrodes provided by the extended electrodes is complemented by allowing the contact admittance to vary spatially or even vanish on some parts of the extended electrodes. The spatial dependence of the contacts is encoded in the so-called {\em boundary admittivity} function. In particular, if the boundary admittivity is smooth over the whole object boundary, there are no singularities at the edges of the electrodes  (cf.~\cite{Hyvonen17b}), which facilitates efficient numerical solution of associated forward problems.

We consider two alternative parametrizations for the boundary admittivity in our two-dimensional numerical experiments: the first one is directly based on the piecewise linear basis functions for the underlying {\em finite element} (FE) discretization, and the other corresponds to hat-shaped surface admittivities with variable positions, widths and heights on the extended electrodes (cf.~\cite{Hyvonen17b}). We demonstrate that both of these choices can be used for incorporating the estimation of the boundary admittivity in a certain iterative Bayesian output least squares reconstruction algorithm for EIT, leading to more accurate reconstructions of the main unknown, i.e.~the interior admittivity, than employing the standard CEM with inaccurate positions for the true electrodes. See \cite{Barber88,Breckon88,Kolehmainen97} for more information on the artifacts caused by geometric mismodeling in EIT.

Incorporating the estimation of the electrode positions in a reconstruction algorithm of EIT has previously been considered in \cite{Candiani19,Darde12,Darde13a,Darde13b,Hyvonen17c}. These  works are based on computing (shape) derivatives of the measurements of EIT with respect to the positions of electrodes with known shapes, not on modeling the uncertainty in the measurement configuration via a boundary admittivity function. Compared to \cite{Candiani19,Darde12,Darde13a,Darde13b,Hyvonen17c}, an attractive aspect of our approach is that estimating the effective locations of the current inputs is conceptually similar to the reconstruction of the interior admittivity, that is, one aims at reconstructing a spatially varying function defined on the employed FE mesh. For the sake of completeness, it should also be mentioned that EIT with imprecisely known measurement setup has previously been considered in \cite{Hyvonen17,Kolehmainen05,Kolehmainen07,Nissinen11,Nissinen11b} to name a few approaches, but none of these or other works on the topic have tackled the uncertainty in the electrode positions via employing a spatially varying boundary admittivity function. Moreover, it is well known that (time or frequency) difference imaging is a partial remedy for geometric mismodeling~\cite{Barber84}, but in this work we are solely working with absolute EIT measurements.

This text is organized as follows. Section~\ref{sec:forward} introduces the new ``electrodeless electrode model'' and tackles its Fr\'echet differentiability with respect to the interior and boundary admittivities. Section~\ref{sec:algorithm}  describes our iterative Bayesian output least squares reconstruction algorithm, considers the two aforementioned alternative parametrizations for the boundary admittivity function and presents a partial convergence result. In Section~\ref{sec:numerics}, the numerical tests based on experimental water tank data are presented. Finally, Section~\ref{sec:conclusions} lists the concluding remarks.

\section{New electrode model for EIT}
\label{sec:forward}

In this section, we first introduce our new electrode model by generalizing the standard CEM and prove its unique solvability. Subsequently, the Fr\'echet differentiability of the associated forward solution with respect to the boundary and domain admittivities is considered.

\subsection{Forward model and its unique solvability}
Let the bounded Lipschitz domain $\Omega \subset \R^n$, $n=2$ or $3$, model the physical body that is under investigation by EIT. Its boundary $\partial \Omega$ is partially covered by $M \in \N \setminus \{ 1 \}$ electrodes $\{ e_m\}_{m=1}^M$ that are identified with the nonempty connected subsets of $\partial \Omega$ that they cover. However, as we do not assume to know the precise locations of the actual electrodes, we introduce so-called extended electrodes $\{ E_m\}_{m=1}^M$ that have the following properties:
\begin{itemize}
\item[(i)] For all $m = 1, \dots, M$,  $e_m \subset E_m \subset \partial \Omega$.
\item[(ii)] The relatively open, connected extended electrodes are mutually disjoint: $E_m \cap E_l = \emptyset$ for all $m \not= l$.
\end{itemize}
In other words, $\{ E_m\}_{m=1}^M$ characterize our prior information on the whereabouts of the true electrodes $\{ e_m\}_{m=1}^M$. Take note that we do not exclude the case $E:= \cup \overbar{E}_m = \partial \Omega$ at this stage.

The contacts at the (extended) electrodes are characterized by the surface admittivity $\zeta: \partial \Omega \to \C$. Since we must be able to feed current through all electrodes and contact conductances with negative real parts are unphysical, the set of admissible admittivities is defined as
\begin{equation}
  \label{eq:zeta}
\mathcal{Z} := \big\{ \zeta \in L^\infty(E) \  \big| \  {\rm Re}( \zeta) \geq 0 \   {\rm and} \  {\rm Re}( \zeta |_{E_m}) \not \equiv 0 \  {\rm for} \ {\rm all} \ m= 1, \dots, M \big\} \subset L^\infty(\partial \Omega),
\end{equation}
where the conditions on $\zeta$ are to be understood in the topology of $L^\infty(E)$. Here and in what follows,  $\mathcal{Z}$ and $L^\infty(E)$ are always interpreted as subsets of $L^\infty(\partial \Omega)$ via zero continuation. On the other hand, the electromagnetic properties of $\Omega$ are characterized by an isotropic admittivity $\sigma: \Omega \to \C$ that belongs to
\begin{equation}
  \label{eq:sigma}
  \mathcal{S} =  \big\{ \sigma \in L^\infty(\Omega) \  \big| \
          {\rm ess} \inf ( {\rm Re}(\sigma))  > 0 \big\} \subset L^\infty(\Omega).
\end{equation}
In other words, the essential infimum of the real part of the conductivity inside $\Omega$ is larger than zero, meaning that current can flow anywhere within $\Omega$.

Suppose the net currents $I_m\in\C $, $m=1, \dots, M$, are driven through the corresponding electrodes and the resulting constant electrode potentials $U_m \in \C$, $m=1, \dots, M$, are measured. Due to conservation of charge, any physically reasonable current pattern $I = [I_1,\dots,I_M]^{\rm T}$ belongs to the subspace
\[
\C^M_\diamond \, := \, \Big\{J \in\C^M\,\Big|\, \sum_{m=1}^M J_m = 0\Big\}.
\]
To facilitate writing down the boundary condition for the CEM, $U = [U_1,\dots,U_M]^{\rm T} \in \C^M$ is identified with
\begin{equation}
\label{eq:piecewise}
U \, = \, \sum_{m=1}^M U_m \chi_m ,
\end{equation}
where $\chi_m$ is the characteristic function of the $m$th {\em extended} electrode $E_m \subset \partial \Omega$. Whether $U$ refers to a piecewise constant function supported on $\overline{E}$ or a vector in $\C^M$ should be clear from the context.

We mimic the formulation of the CEM presented in \cite{Hyvonen17b}: The electromagnetic potential $u$ inside $\Omega$ and the piecewise constant electrode potential $U$ jointly satisfy
\begin{equation}
\label{eq:cemeqs}
\begin{array}{ll}
\displaystyle{\nabla \cdot(\sigma\nabla u) = 0 \qquad}  &{\rm in}\;\; \Omega, \\[6pt]
{\displaystyle {\nu\cdot\sigma\nabla u} = \zeta (U - u) } \qquad &{\rm on}\;\; \partial \Omega, \\[2pt]
{\displaystyle \int_{E_m}\nu\cdot\sigma\nabla u\,{\rm d}S} = I_m, \qquad & m=1,\ldots,M, \\[4pt]
\end{array}
\end{equation}
interpreted in the weak sense and with $\nu \in L^\infty(\partial \Omega, \R^n)$ denoting the exterior unit normal of $\partial\Omega$. Unlike in \cite{Hyvonen17b}, the true electrodes do not appear in the forward problem \eqref{eq:cemeqs} that is formulated using the extended electrodes. However, if the positions of the true electrodes were known, then one could return to the setting analyzed in \cite{Hyvonen17b} by simply defining $\zeta = 0$ in $E_m \setminus \overbar{e}_m$ for all $m=1, \dots, M$. The connection between the original formulation of the CEM in \cite{Cheng89,Somersalo92} and \eqref{eq:cemeqs} is discussed in \cite{Hyvonen17b}; to put it short, the contact impedances in the traditional formulation of the CEM are obtained as $z_m := (1/\zeta)|_{E_m}$, $m=1, \dots, M$, if $e_m = E_m$ and spatially varying contacts, which may have singularities, are allowed.

We seek the solution of \eqref{eq:cemeqs} in the quotient space $\mathcal{H}^1$ defined via
\begin{align*}
  \mathcal{H}^s &:= \big\{ \{ (v + c, V + c {\bf 1}) \, | \, c \in \C \} \, \big| \, (v, V) \in H^s(\Omega)\oplus \C^M \big\}, \qquad s \in \R,
\end{align*}
with ${\bf 1} := [1, \dots, 1]^{\rm T}\in \C^M$. In other words, all elements of $H^s(\Omega)\oplus \C^M$ that differ by an additive constant are identified as an equivalence class --- this corresponds to the freedom in the choice of the common ground level of potential for the interior and electrode potentials. In particular, when the second component of an element in $\mathcal{H}^s$ is interpreted as a piecewise constant function on the electrodes, the additive constant is also supported on $\overline{E}$. The standard quotient norm for $\mathcal{H}^s$ is defined as
\begin{equation}
\label{eq:norm}
\|(v,V)\|_{\mathcal{H}^s} := \inf_{c\in\C}\Big( \|v-c\|_{H^s(\Omega)}^2 + \| V - c {\bf 1}\|_2^2 \Big)^{1/2},
\end{equation}
where $\| \cdot  \|_2$ denotes the Euclidean norm. We use the notation $\| \cdot  \|_{\C^M / \C}$ for the associated quotient norm of $\C^M$ defined by dropping out the first term on the right-hand side of \eqref{eq:norm}.

As can easily be deduced based on material in \cite{Hyvonen17b,Somersalo92}, the variational formulation of \eqref{eq:cemeqs} amounts to finding $(u,U) \in \mathcal{H}^1$ such that
\begin{equation}
\label{eq:weak}
B_{\sigma,\zeta}\big((u,U),(v,V)\big)  \,=  \, I\cdot \overbar{V} \qquad {\rm for} \ {\rm all} \ (v,V) \in  \mathcal{H}^1,
\end{equation}
where the sesquilinear form $B_{\sigma,\zeta}: \mathcal{H}^1 \times \mathcal{H}^1 \to \C$ is defined by
\begin{equation}
\label{eq:sesqui}
B_{\sigma,\zeta}\big((w,W),(v,V)\big) = \int_\Omega \sigma\nabla w\cdot \nabla \overbar{v} \,{\rm d}x + \int_{\partial \Omega} \zeta (W-w)(\overbar{V}-\overbar{v})\,{\rm d}S,
\end{equation}
with $W, V \in \C^M$ identified with the corresponding piecewise constant functions. If $\sigma \in \mathcal{S}$ and $\zeta \in \mathcal{Z}$ are admissible, the sesquilinear form $B_{\sigma,\zeta}: \mathcal{H}^1 \times \mathcal{H}^1 \to \C$ is bounded and coercive:
\begin{align}
  \label{eq:cont}
  \big| B_{\sigma,\zeta}\big((w,W), (v,V) \big) \big| \, &\leq \, C \|(w,W) \|_{\mathcal{H}^1}  \|(v,V) \|_{\mathcal{H}^1}, \\[1mm]
  \label{eq:coer}
{\rm Re} \Big( B_{\sigma,\zeta} \big((v,V), (v,V) \big) \Big) \, &\geq \, c \|(v,V) \|_{\mathcal{H}^1}^2,
\end{align}
where $c=c(\sigma,\zeta, \Omega,E)>0$, $C=(\sigma,\zeta, \Omega,E) > 0$ are independent of $(w,W), (v,V) \in \mathcal{H}^1$. This conclusion follows by repeating the argumentation in the proof of \cite[Lemma~2.1]{Hyvonen17b}, and it also leads to the following theorem that is essentially a restatement of \cite[Theorem~2.2]{Hyvonen17b} with the extended electrodes playing here the role of the actual electrodes in \cite{Hyvonen17b}.

\begin{theorem}
\label{thm:existence}
If $\sigma \in \mathcal{S}$ and $\zeta \in \mathcal{Z}$, the problem \eqref{eq:cemeqs} has a unique solution $(u,U) \in \mathcal{H}^1$ for any
current pattern $I \in \C_\diamond^M$. Moreover,
$$
\| (u,U) \|_{\mathcal{H}^1} \leq C \| I \|_2,
$$
where $C= C(\Omega,E,\sigma,\zeta)>0$ is independent of $I$.
\end{theorem}

If the extended electrodes are well separated, the regularity of $u$ is dictated by the smoothness of $\zeta \in \mathcal{Z} \subset L^\infty(\Omega)$ interpreted via zero continuation as a function on the whole of $\partial \Omega$ together with the regularity of the interior admittivity $\sigma$ and the boundary $\partial \Omega$. In particular, smooth parametrizations for $\zeta$ and $\sigma$ facilitate efficient numerical solution of \eqref{eq:cemeqs}, say, by resorting to higher order FEs~\cite{Hyvonen17b}. For simplicity,  the following theorem is stated assuming $\partial \Omega$ and $\sigma$ are infinitely smooth; the effect that irregularities in $\sigma$ and $\partial \Omega$ have on the regularity of $u$ could be straightforwardly analyzed based on standard elliptic theory (see,~e.g.,~\cite{Necas12}).

\begin{theorem}
\label{thm:smoothness}
Assume that $\sigma \in \mathcal{S} \cap \mathcal{C}^\infty(\overline{\Omega})$, $\partial \Omega$ is of class $\mathcal{C}^\infty$ and $\overbar{E}_m \cap \overbar{E}_l = \emptyset$ for $m \not = l$. If $\zeta \in \mathcal{Z} \cap H^s(\partial \Omega)$ for some $s > (n-1)/2$, then the solution to \eqref{eq:weak} satisfies
\begin{equation}
\label{scont}
\| (u,U) \|_{\mathcal{H}^{s+3/2}} \leq C  \| I \|_2,
\end{equation}
where $C = C(\Omega, E, \sigma, \zeta, s) > 0$ is independent of $I \in \C_\diamond^M$.
\end{theorem}

\begin{proof}
  The claim is a consequence of standard elliptic regularity theory, and it can be deduced by following the line of reasoning leading to \cite[Theorem~2.5]{Hyvonen17b}.
\end{proof}

\begin{remark}
  The essential ingredient of Theorem~\ref{thm:smoothness} is that the smooth behavior of $\zeta$ compensates to a certain extent for the singularities caused by the piecewise constant nature of  $U$ on the right-hand side of the second condition in \eqref{eq:cemeqs}. This leads to the possibility to iterate a bootstrap argument that the Neumann trace of $u$ has the same Sobolev regularity as its Dirichlet trace, which in turn leads to higher regularity for $u$ itself up to a certain limit dictated by $\zeta$. It seems intuitive that one could  use such an argument also in the case when the closures of the extended electrodes are allowed to touch if it is required that $\zeta$ goes smoothly enough to zero at their common boundaries. However, as the case of touching (extended) electrodes was not considered in \cite{Hyvonen17b}, we do not stress this matter any further to avoid rewriting the corresponding proofs.
\end{remark}

\subsection{Extension for unphysical boundary conductivities}
Before considering the differentiability of the second part of the solution to \eqref{eq:cemeqs} with respect to $\zeta$, it is natural to extend the set of admissible boundary conductivities so that it becomes an open subset of $L^\infty(E)$. To this end, we simply define
\begin{equation}
  \label{eq:zeta2}
\mathcal{Z}' = \big\{ \zeta \in L^\infty(E) \ \big | \ \text{For any $\sigma \in \mathcal{S}$, the condition \eqref{eq:coer} holds with $c = c(\sigma,\zeta) > 0$.} \big\},
\end{equation}
which is, as usual, interpreted via zero continuation as a subset of $L^\infty(\partial \Omega)$. Since it is clear that \eqref{eq:cont} actually holds for any $\sigma \in L^\infty(\Omega)$ and $\zeta \in L^\infty(E)$, it follows immediately from the Lax--Milgram lemma that Theorem~\ref{thm:existence} remains valid if $\mathcal{Z}$ is replaced by $\mathcal{Z}'$ in its assertion. Moreover, $\mathcal{Z}'$ is indeed open as a subset of $L^\infty(E)$, which is a byproduct of the following more general lemma.

\begin{lemma}
  \label{lemma:coer}
  For any $(\sigma,\zeta) \in \mathcal{S} \times \mathcal{Z}'$, there exists $\delta = \delta(\sigma,\zeta,\Omega,E) > 0$ such that
  $$
  {\rm Re} \Big( B_{\sigma + \omega,\zeta + \eta} \big((v,V), (v,V) \big) \Big) \geq \, \frac{c(\sigma,\zeta)}{2} \|(v,V) \|_{\mathcal{H}^1}^2 \qquad \forall (v,V) \in \mathcal{H}^1
  $$
  for all $(\omega,\eta) \in L^\infty(\Omega) \times L^\infty(E)$ satisfying $\|\omega \|_{L^\infty(\Omega)}, \|\eta \|_{L^\infty(E)} \leq \delta$. In particular, $\mathcal{Z}'$ is an open subset of $L^\infty(E)$.
\end{lemma}

\begin{proof}
  Consider arbitrary $(\sigma,\zeta) \in \mathcal{S} \times \mathcal{Z}'$ and $(\omega, \eta) \in L^\infty(\Omega) \times L^\infty(E)$. For any $(v,V) \in \mathcal{H}^1$,
\begin{align*}
  {\rm Re} \Big( B_{\sigma+\omega,\zeta+\eta} \big((v,V), (v,V) \big) \Big) & = {\rm Re} \Big( B_{\sigma,\zeta} \big((v,V), (v,V) \big) \Big) +  \int_{\Omega}{\rm Re} (\omega) |\nabla v|^2 \,{\rm d}x \\
&  \qquad \quad +  \int_{\partial \Omega} {\rm Re}(\eta) | V- v|^2 \, {\rm d} S \\[1mm]
& \geq \big( c(\sigma, \zeta) - \| \omega\|_{L^\infty(\Omega)}  - K \| \eta \|_{L^\infty(E)} \big)\| (v,V) \|_{\mathcal{H}^1}^2,
\end{align*}
where we used \eqref{eq:zeta2}, $K = K(\Omega,E)>0$, and the estimate for the integral over $\partial \Omega$ follows from the same reasoning as~\cite[Lemma~2.4]{Hyvonen17b}. Choosing $\delta = c(\sigma, \zeta)/(4 \max \{1,K\})$ completes the proof.
\end{proof}

 In particular, the forward problem \eqref{eq:cemeqs} remains uniquely solvable if $\zeta \in \mathcal{Z} \subset \mathcal{Z}'$ is replaced by $\zeta + \eta$, with a small enough $0 \not=\eta \in L^\infty(E)$, even if $\zeta + \eta$ is negative in its real part on some subsets of $E$ (cf.~\eqref{eq:zeta}).

\subsection{Fr\'echet derivative with respect to the boundary conductivity}

The electrode potentials, i.e.~the second part of the solution to \eqref{eq:cemeqs}, can obviously be considered as a function of three variables:
\begin{equation}
  \label{eq:meas_map}
U:
\left\{
\begin{array}{l}
  (\sigma, \zeta, I) \mapsto U(\sigma, \zeta, I), \\[1mm]
  \mathcal{S} \times \mathcal{Z}' \times \C^M_\diamond \to \C^M_\diamond,
\end{array}
\right.
\end{equation}
where we have systematically selected the ground level of potential so that $U$ has always zero mean. Since our choice for the (extended) set of admissible boundary conductivities $\mathcal{Z}'$ is nonstandard (cf.~\cite{Hyvonen18}), we need to verify some basic properties of the measurement map \eqref{eq:meas_map}.

\begin{lemma}
  \label{lemma:cont}
  The mapping $U: \mathcal{S} \times \mathcal{Z}' \times \C^M_\diamond \to \C^M_\diamond$ is continuous.
\end{lemma}

\begin{proof}
  Let $(\sigma, \zeta, I) \in \mathcal{S} \times \mathcal{Z}' \times \C^M_\diamond$ be arbitrary and the first two components of $(\omega,\eta,J) \in L^\infty(\Omega) \times L^\infty(E) \times \C^M_\diamond$ small enough so that also $(\sigma + \omega, \zeta + \eta, I+J) \in \mathcal{S} \times \mathcal{Z}' \times \C^M_\diamond$; see Lemma~\ref{lemma:coer} and \eqref{eq:sigma}. Let $(u,U)$ and $(u_{\omega,\eta,J},U_{\omega,\eta,J})$ be the unique solutions of \eqref{eq:cemeqs} for the parameter triplets $(\sigma, \zeta, I)$ and $(\sigma + \omega, \zeta + \eta, I+J)$, respectively, and define $(w_{\omega,\eta,J},W_{\omega,\eta,J}) = (u_{\omega,\eta,J} - u ,U_{\omega,\eta,J}-U)$. By subtracting the corresponding variational formulations~\eqref{eq:sesqui}, one easily deduces that
\begin{equation}
  \label{eq:diff_var}
  B_{\sigma+\omega,\zeta+\eta}\big((w_{\omega, \eta,J},W_{\omega,\eta,J}),(v,V)\big) = J\cdot \overline{V} - \int_\Omega \omega \nabla u \cdot \nabla \overbar{v} \,{\rm d}x - \int_{\partial \Omega} \eta (U -u)(\overbar{V}-\overbar{v})\,{\rm d}S
  \end{equation}
  for all $(v,V) \in \mathcal{H}^1$.

 Assume that $\|\omega \|_{L^\infty(\Omega)}, \|\eta \|_{L^\infty(E)} \leq \delta$, with $\delta = \delta(\sigma,\zeta) >0$ as in  Lemma~\ref{lemma:coer}. Choosing $(v,V) = (w_{\omega,\eta,J},W_{\omega,\eta,J})$ in \eqref{eq:diff_var} and employing Lemma~\ref{lemma:coer} yields
  \begin{align*}
    \|(w_{\omega,\eta,J}, &W_{\omega,\eta,J}) \|_{\mathcal{H}^1}^2  \leq \frac{2}{c(\sigma, \zeta)} \Big(  \|J\|_2 \| W_{\omega,\eta,J}\|_{\C^M/\C} + \int_{\Omega} |\omega| | \nabla u \cdot \nabla \overline{w}_{\omega,\eta,J} | \, {\rm d} x  \\
    & \qquad \qquad \qquad  \qquad \qquad  +   \int_{\partial \Omega} |\eta| |U - u| | W_{\omega,\eta,J} - w_{\omega,\eta,J}  | \, {\rm d} S\Big) \\[1mm]
    &\leq C'\big(\|J\|_2 +  \max \{ \|\omega\|_{L^\infty(\Omega)},  \|\eta\|_{L^\infty(E)} \} \|(u,U)\|_{\mathcal{H}^1} \big) \|(w_{\omega,\eta,J}, W_{\omega,\eta,J})\|_{\mathcal{H}^1},
  \end{align*}
  where $C' = C'(\sigma,\zeta, \Omega,E) > 0$ and the integral over $\partial \Omega$ was handled by resorting to similar means as in the proof of Lemma~\ref{lemma:coer}. Dividing by $\|(w_{\omega,\eta,J}, W_{\omega,\eta,J})\|_{\mathcal{H}^1}$ and applying the generalization of Theorem~\ref{thm:existence} mentioned before Lemma~\ref{lemma:coer}, we finally obtain
  \begin{equation}
    \label{eq:contiunity}
\|W_{\omega,\eta,J} \|_{\C^M / \C} \! \leq \| (w_{\omega,\eta,J}, W_{\omega,\eta,J}) \|_{\mathcal{H}^1} \! \leq C'' \big( \|J\|_2 + \max \{ \|\omega\|_{L^\infty(\Omega)},  \|\eta\|_{L^\infty(E)} \} \| I \|_2 \big),
\end{equation}
where $C'' = C''(\sigma,\zeta, \Omega,E) > 0$. Because the quotient norm $\| \cdot \|_{\C^M / \C}$ is equivalent to $\| \cdot \|_2$ on $\C_{\diamond}^M$, the proof is complete.
\end{proof}

Let us next prove the differentiability of $U: \mathcal{S} \times \mathcal{Z}' \times \C^M_\diamond \to \C^M_\diamond$; such a result is of course well known for the standard formulation of the CEM~(see,~e.g.,~\cite{Hyvonen18,Lechleiter06}), but our extended definition for the admissible contact conductivities arguably merits (re)writing a formal proof. To this end, let us introduce an auxiliary variational problem of finding $(u',U') = (u'(\sigma, \zeta, I; \omega, \eta, J), U'(\sigma, \zeta, I; \omega, \eta, J)) \in \mathcal{H}^1$ such that
\begin{equation}
  \label{eq:zetaD}
 B_{\sigma,\zeta}\big((u',U'), (v,V) \big) = J \cdot \overline{V} - \int_\Omega \omega \nabla u \cdot \nabla \overbar{v} \, {\rm d} x - \int_{\partial \Omega}\eta \, (U - u) (\overbar{V} - \overbar{v}) \, {\rm d} S \quad \ \forall (v,V) \in \mathcal{H}^1,
\end{equation}
where $\sigma \in \mathcal{S}$, $\zeta \in \mathcal{Z}'$, $J \in \C_\diamond^M$, $\omega \in L^\infty(\Omega)$, $\eta \in L^\infty(E) \subset L^\infty(\partial \Omega)$, and $(u,U) = (u(\sigma, \zeta, I),U(\sigma, \zeta, I)) \in \mathcal{H}^1$ is the solution to \eqref{eq:cemeqs}. Since it is straightforward to check that the right-hand side of \eqref{eq:zetaD} satisfies (cf.,~e.g.,~\cite{Hyvonen17b})
\begin{align*}
  \Big| J \cdot \overline{V} - \int_\Omega \omega \nabla u & \cdot \nabla \overbar{v} \, {\rm d} x  - \int_{\partial \Omega}\eta \, (U - u) (\overbar{V} - \overbar{v}) \, {\rm d} S \Big| \\[1mm]
  & \leq C \big( \|J\|_2 + \max\{ \| \omega \|_{L^\infty(\Omega)},\|\eta\|_{L^\infty(E)}\} \|(u,U)\|_{\mathcal{H}^1}\big)  \|(v,V)\|_{\mathcal{H}^1},
\end{align*}
it follows from the Lax--Milgram theorem that \eqref{eq:zetaD} has a unique solution that satisfies
\begin{align*}
  \| (u',U') \|_{\mathcal{H}^1} &\leq C \big(  \|J\|_2 + \max\{ \| \omega \|_{L^\infty(\Omega)},\|\eta\|_{L^\infty(E)}\} \|(u,U)\|_{\mathcal{H}^1}\big)  \\[1mm]
  &\leq C' \big(  \|J\|_2 + \max\{ \| \omega \|_{L^\infty(\Omega)},\|\eta\|_{L^\infty(E)}\} \|I \|_{2}\big),
\end{align*}
where $C= C(\Omega, E, \sigma, \zeta)$ and $C'= C'(\Omega, E, \sigma, \zeta)$ are positive constants. Moreover, the pair $(u',U')$ depends linearly on $(J, \omega, \eta)$ as does the right-hand side of \eqref{eq:zetaD}.

\begin{theorem}
  \label{thm:U-zetadiff}
  The mapping
  $$
  \mathcal{S} \times \mathcal{Z}' \times \C_\diamond^M \ni (\sigma,\zeta,I) \mapsto U(\sigma, \zeta, I) \in \C_\diamond^M
  $$
  is Fr\'echet differentiable. To be more precise, for any $(\sigma,\zeta,I)\in
  \mathcal{S} \times \mathcal{Z}' \times \C_\diamond^M$,
  \begin{align*}
    \big\| U(&\sigma+\omega, \zeta + \eta, I+J) - U(\sigma, \zeta, I)  - U'(\sigma, \zeta, I; \omega,\eta,J) \big\|_2 \\[1mm]
    &= \mathcal{O}\big(\max\{ \| \omega \|_{L^\infty(\Omega)}^2,\| \eta \|_{L^\infty(E)}^2, \|J\|_2^2 \} \big), \quad \ \ (\omega, \eta,J) \in L^\infty(\Omega) \times L^\infty(E) \times \C_\diamond^M,
  \end{align*}
  where the Fr\'echet derivative $U'(\sigma, \zeta, I; \omega,\eta,J) \in \C_\diamond^M$ is the second part of the solution to \eqref{eq:zetaD}. Moreover, this derivative can be assembled through the relation
  \begin{equation}
    \label{eq:zeta_sampling}
U'(\sigma, \zeta, I; \omega, \eta,J) \cdot \tilde{I} = J \cdot \tilde{U} - \int_{\Omega} \omega \nabla u \cdot \nabla \tilde{u} \, {\rm d} x -  \int_{\partial \Omega} \eta \, (U - u) (\tilde{U} - \tilde{u}) \, {\rm d} S,
\end{equation}
  where $(U, u), (\tilde{U}, \tilde{u}) \in \mathcal{H}^1$ are the solutions for \eqref{eq:cemeqs} at the common parameters
  $(\sigma, \zeta) \in \mathcal{S} \times \mathcal{Z}'$ but for the current patterns $I, \tilde{I} \in \C_\diamond^M$, respectively.
\end{theorem}

\begin{proof}
  Let us adopt the notation introduced in the proof of Lemma \ref{lemma:cont}.  Subtracting \eqref{eq:zetaD} from \eqref{eq:diff_var} results in
  \begin{align*}
   B_{\sigma,\zeta}\big((w_{\omega, \eta,J}-u',&W_{\omega,\eta,J}-U'),(v,V)\big) =  \\ &- \int_\Omega \omega \nabla w_{\omega, \eta,J} \cdot \nabla \overbar{v} \, {\rm d} x - \int_{\partial \Omega}\eta \, (W_{\omega, \eta,J} - w_{\omega, \eta,J}) (\overbar{V} - \overbar{v}) \, {\rm d} S \qquad
  \end{align*}
  for all $(v,V) \in \mathcal{H}^1$. Setting $(v,V) = (w_{\omega, \eta,J}-u', W_{\omega,\eta,J}-U')$ and employing \eqref{eq:coer} (cf.~\eqref{eq:zeta2}), it is straightforward to deduce that
  \begin{align*}
    \| (w_{\omega, \eta,J}-u', W_{\omega,\eta,J}-U') \|_{\mathcal{H}^1}^2 &\leq C \max\{\| \omega \|_{L^\infty(\Omega)}, \| \eta \|_{L^\infty(E)}\} \, \| (w_{\omega, \eta,J}, W_{\omega,\eta,J}) \|_{\mathcal{H}^1}  \\[1mm]
    & \qquad \qquad \qquad \times \| (w_{\omega, \eta,J}-u', W_{\omega,\eta,J}-U') \|_{\mathcal{H}^1},
  \end{align*}
  where $C = C(\sigma, \zeta, \Omega, E) >0$. After dividing by $\| (w_{\omega, \eta,J}-u', W_{\omega,\eta,J}-U') \|_{\mathcal{H}^1}$, the claim on the Fr\'echet differentiability follows from the second inequality in \eqref{eq:contiunity} under the assumption that $\| \omega \|_{L^\infty(\Omega)}, \| \eta \|_{L^\infty(E)} < \delta$, with $\delta = \delta(\sigma, \zeta)>0$ as in Lemma~\ref{lemma:coer}.

  The representation formula \eqref{eq:zeta_sampling} can  be deduced by using the complex conjugate  $(\overline{\tilde{U}}, \overline{\tilde{u}})$ of the solution to the original problem \eqref{eq:weak} for $\tilde{I} \in \C_\diamond^M$ as the test function pair in \eqref{eq:zetaD} and subsequently equating the left-hand sides of \eqref{eq:weak} and \eqref{eq:zetaD}.
\end{proof}

\begin{remark}
	\label{rem:paramD}
	By the chain rule for Banach spaces (see,~e.g.,~\cite[Theorem~7.1-3]{Ciarlet13}), the representation formula \eqref{eq:zeta_sampling} also extends to an arbitrary parametrization of $\zeta$ as
	$$
        \partial_\theta U(\sigma, \zeta(\theta), I; \vartheta) \cdot \tilde{I} = - \int_{\partial \Omega} \partial_\theta \zeta(\theta;\vartheta) \, (U - u) (\tilde{U} - \tilde{u}) \, {\rm d} S, \qquad \theta \in D, \ \vartheta \in \mathcal{B},
        $$
        assuming the mapping $\mathcal{B} \supset D \ni \theta \mapsto \zeta(\theta) \in \mathcal{Z}' \subset L^\infty(E)$ is Fr\'echet differentiable, with $\mathcal{B}$ being a Banach space and $D$ its open subset. A similar conclusion also applies to parametrizations of the domain conductivity.
\end{remark}

\section{Bayesian reconstruction algorithm}
\label{sec:algorithm}
In this section we briefly describe our numerical forward solver and formulate a Gauss--Newton type reconstruction algorithm for simultaneous reconstruction of the isotropic domain admittivity $\sigma$ and the boundary admittivity $\zeta$. Since both $\sigma$ and $\zeta$ are modeled as real-valued in our numerical experiments, they will be referred to as the domain and surface conductivities in what follows. Similar algorithms have been previously employed for related problems in,~e.g.,~\cite{Hyvonen17c, Vilhunen02}.

To facilitate a reasonable numerical implementation, we assume that the domain $\Omega \subset \R^2$ is simply connected and sufficiently smooth, and that it can be easily triangulated. Furthermore, we assume that the boundary subsections corresponding to the extended electrodes $\{E_m\}_{m=1}^M \subset \partial\Omega$ are well separated and connected. However, we still only assume that $e_m \subset E_m$ for each $m = 1, \dots, M$,~i.e.,~that each extended electrode covers the corresponding true electrode, but it does not usually hold that $e_m = E_m$. Finally, we assume that the extended electrodes are enumerated and parametrized along the boundary curve $\partial \Omega$.

Approximate solutions for \eqref{eq:cemeqs} are computed using a \emph{finite element method} (FEM). In our FEM solver, the domain $\Omega$ is discretized into a triangle mesh, denoted by $\Omega_h$, on which the domain and boundary potentials are solved as linear combinations of piecewise linear basis functions. The domain conductivity $\sigma \in L^{\infty}_+(\Omega_h, \R)$ is discretized with respect to the same piecewise linear basis as the solution. Inspired by the observations in \cite{Hyvonen18}, our algorithm actually aims at reconstructing the (finite-dimensional) log-conductivity $\kappa := \log \sigma \in L^{\infty}(\Omega_h, \R)$ instead of the actual conductivity.\footnote{However, the conductivity $\sigma = \exp{\kappa}$ itself is visualized in the presented reconstructions to facilitate comparison with previous works.} The derivatives of the forward solutions with respect $\kappa$ can be computed by resorting to \eqref{eq:zeta_sampling} and the chain rule for Banach spaces; see Remark~\ref{rem:paramD} and \cite[Eq.~(9)]{Hyvonen18}.

The piecewise linear boundary curve $\partial \Omega_h$ naturally also determines the edges on which the extended electrodes are defined. In order to avoid current leaking directly between electrodes, we require at least one boundary node (and the corresponding basis function) in-between any adjacent extended electrodes on $\partial \Omega_h$.
The contact conductivity $\zeta$ is parametrized on the piecewise linear boundary curve of the mesh. We formulate two different parametrizations for $\zeta$ in Section \ref{sec:zetaoptions}, denoted generally by a finite-dimensional real map $\theta \mapsto \zeta(\theta)$.

In order to simplify the notation, we define concatenated vectors for voltage measurements, current injections and forward model solutions, respectively, as
$$
\begin{aligned}
	\mathcal{V} &= \begin{bmatrix} (V^{(1)})^\top, \dots , (V^{(M-1)})^\top  \end{bmatrix}^\top \in \R^{M(M-1)},\\[1mm]
	\mathcal{I} &= \begin{bmatrix} (I^{(1)})^\top, \dots , (I^{(M-1)})^\top  \end{bmatrix}^\top \in \R^{M(M-1)} \quad \text{and} \\[1mm]
	\mathcal{U}(\kappa, \theta, \mathcal{I}) &= \begin{bmatrix} (U(\exp(\kappa), \zeta(\theta), I^{(1)}))^\top, \dots, (U(\exp(\kappa), \zeta(\theta), I^{(M-1)}))^\top  \end{bmatrix}^\top \in \R^{M(M-1)},
\end{aligned}
$$
where the measured (noisy) electrode voltages $\{V^{(i)}\}_{i=1}^{M-1} \subset \R^M_\diamond$ correspond to the current patterns $\{I^{(i)}\}_{i=1}^{M-1} \subset \R^M_\diamond$, and $\{U(\exp(\kappa), \zeta(\theta), I^{(i)})\}_{i=1}^{M-1}$ are the electrode potential components of the solutions to the forward problem \eqref{eq:cemeqs} for given domain log-conductivity $\kappa$, parametrized boundary conductivity $\zeta(\theta)$ and current pattern $I^{(i)}$.  The `current basis' for $\R_\diamond^{M-1}$ is defined via $I^{(i)} = c(e_1 - e_{i+1})$ with a physically reasonable $c \in \R$ and $e_1, \dots, e_m \in \R^M$ denoting the standard basis vectors, that is, the first electrode feeds the current into the body and it is guided out in turns through the other $M-1$ electrodes. In what follows, we do not explicitly write down the dependence of $\mathcal{U}(\kappa, \theta, \mathcal{I})$ on $\mathcal{I}$ to shorten the notation.

We formulate the reconstruction problem as minimization of a Tikhonov-type functional:
\begin{equation}
	\label{eq:minimization}
	\arg\min_{\kappa, \theta} \left( \| \mathcal{U}(\kappa, \theta) - \mathcal{V} \|^2_{\Gamma^{-1}_{\rm noise}} + \| \kappa - \kappa^\mu \|^2_{\Gamma^{-1}_\kappa} + \| \theta - \theta^\mu \|^2_{\Gamma^{-1}_\theta} \right) ,
\end{equation}
where the notation $\| x \|^2_A := x^\top \! A x$ is used. We have included in \eqref{eq:minimization} prior information motivated by the Bayesian inversion paradigm: the symmetric positive definite covariance matrices $\Gamma_{\rm noise}$, $\Gamma_\kappa$ and $\Gamma_\theta$ correspond to, respectively, the measurement noise, the domain log-conductivity and the parametrization of the boundary conductivity, while $\kappa^\mu$ and $\theta^\mu$ are the expected values for the to-be-estimated variables. To be more precise, \eqref{eq:minimization} is equivalent to finding a {\em maximum a posteriori} estimate for the parameters defining the domain and boundary conductivities assuming that they and the additive measurement noise follow {\em a priori} the Gaussian distributions $\mathcal{N}(\kappa^\mu, \Gamma_\kappa)$, $\mathcal{N}(\theta^\mu, \Gamma_\theta)$ and $\mathcal{N}(0, \Gamma_{\rm noise})$, respectively; see,~e.g.,~\cite{Kaipio05}. As the covariance matrices are assumed positive definite, the minimization functional is bounded from below.

For the piecewise linear parametrization of the domain conductivity, we always use a Gaussian smoothness prior with the covariance structure
\begin{equation}
  \label{eq:cov_kappa}
	(\Gamma_\kappa)_{ij} = \gamma^2_\kappa \exp \left( -\frac{|x_i - x_j|^2}{2 \lambda^2_\kappa} \right),
\end{equation}
where $x_i, x_j \in \R^2$ are nodal positions in the FE mesh of $\Omega_h$, $\gamma^2_\kappa$ is the pointwise variance, and $\lambda_\kappa$ is the correlation length controlling the expected spatial smoothness of the reconstructed admittivity. The employed forms for $\Gamma_\theta$ are considered in the following subsection where our parametrizations for the boundary conductivity are described.

We aim at solving \eqref{eq:minimization} iteratively using the Gauss--Newton algorithm~\cite{Nocedal99}. Let $J(\kappa, \theta) = \begin{bmatrix} \partial_\kappa\mathcal{U}(\kappa, \theta), \partial_\theta\mathcal{U}(\kappa, \theta) \end{bmatrix}$ be the Jacobian matrix of $\mathcal{U}(\kappa, \theta)$ with respect to the discretized domain log-conductivity $\kappa$ and the parameters defining the boundary conductivity $\theta$. The Jacobian matrix can be assembled using the representation formula for the Fr\'echet derivative~\eqref{eq:zeta_sampling}. See Section~\ref{sec:zetaoptions} below for information on the employed parametrization $\theta \mapsto \zeta(\theta)$ as well as on using \eqref{eq:zeta_sampling} in connection to those parametrizations.

It is straightforward to see that applying the basic form of the Gauss--Newton algorithm to \eqref{eq:minimization} corresponds to iteratively linearizing $\mathcal{U}(\kappa, \theta)$ in the first term of the to-be-minimized functional in \eqref{eq:minimization} around the current estimate for the solution and subsequently defining the next estimate to be the minimizer of the resulting quadratic Tikhonov functional. To this end, observe that the aforementioned linearization around a given point $\tau := (\kappa, \theta)$ leads via slight abuse of  notation to the least squares problem
\begin{align}
	\label{eq:gn}
	\arg\min_{\Delta\tau} &
	\left\| L \left(
	\begin{bmatrix}
		J(\tau) \\[1mm]
		\mathrm{I}
	\end{bmatrix}
        \!\Delta \tau
	-
	\begin{bmatrix}
		\mathcal{U}(\tau) - \mathcal{V}  \\[1mm]
		\tau - \tau^\mu \\
	\end{bmatrix}
	\right) \right\|^2,
\end{align}
where $\Delta \tau$ denotes the additive change in $\tau$, $\tau^\mu = (\kappa^\mu, \theta^\mu)$ is the expected value for the total parameter vector and $\mathrm{I}$ is an identity matrix of the appropriate size. Furthermore, $L$ is a Cholesky factor for the block diagonal concatenation of the inverses of the covariance matrices $\Gamma_{\rm noise}$, $\Gamma_\kappa$ and $\Gamma_\theta$,~i.e.~$L^\top L = \diag(\Gamma^{-1}_{\rm noise}, \Gamma^{-1}_\kappa, \Gamma^{-1}_\theta)$.

Accompanying the basic Gauss--Newton algorithm with line search to deduce an optimal step size for each update, we finally end up with the following algorithm:

\bigskip

\begin{algorithm}[H]
	Set iteration count to $j = 0$. \\
	Choose an initial guess $\tau_0 = (\kappa_0, \theta_0)$. \\
	\While{Stopping criteria not met}{
		Compute $J_j = J(\tau_j)$\\
		Determine the step direction $d_j$
                as the least squares solution of
		$$
			L \begin{bmatrix} J_j \\[1mm] \mathrm{I} \end{bmatrix} \! d_j =
			L \begin{bmatrix} \mathcal{U}(\tau_j) - \mathcal{V} \\[1mm] \tau_j - \tau^\mu  \end{bmatrix}.
		$$\\
		        Using a backtracking line search, seek a step length $t \in [0, 1]$ such that $\tau_j + t d_j$ obtains sufficient decrease on the original Tikhonov functional in \eqref{eq:minimization}; compare \cref{thm:convergence}. \\
     		Update $\tau_{j+1} = \tau_j + t_j  d_j$. \\
		Update $j = j + 1$.
	}
	\caption{Simultaneous reconstruction of $\kappa$ and $\theta$}
	\label{alg:minimization}
\end{algorithm}

\bigskip

\subsection{Two alternative parametrizations for the boundary conductivity~$\zeta$}
\label{sec:zetaoptions}

In the standard CEM the parametrization of the boundary conductivity can be interpreted as a piecewise constant function defined on the boundary of the domain
\begin{equation*}
	\zeta_{\rm CEM}(x) = \sum\limits_{m = 1}^M \theta_m \tilde{\chi}_m(x), \qquad \theta_m \in \R_+, \ x \in \partial \Omega,
\end{equation*}
where $\tilde{\chi}_m$ is the characteristic function of the true electrode $e_m$ on $\partial \Omega$. Although this model is known to be precise enough to match measurement accuracy \cite{Cheng89,Somersalo92}, it has some often ignored inconveniences. First, if the exact locations of the electrode contacts are not known, significant reconstruction errors can be expected \cite{Barber88,Breckon88,Darde12,Kolehmainen97}. Second, the piecewise constant nature of the contact conductivity (or resistivity) has a negative effect on the smoothness of the solution to \eqref{eq:cemeqs}, which in turn decreases the highest attainable convergence rate for, say, higher order FEMs when they are applied to numerically solving~\eqref{eq:cemeqs} or evaluating \eqref{eq:zeta_sampling} \cite{Hyvonen17b}.
Finally, there seems to be no other (physical) justification for a constant boundary conductivity over each electrode except the Occam's razor: the simplest model is the most attractive one. In reality, there is no obvious reason to expect equally good contact across the whole electrode, especially if the contact is established by simply pressing a metal electrode tip against skin in some medical application of EIT.

Our first novel parametrization for the boundary conductivity corresponds to simply using the piecewise linear FE basis restricted to the boundary. This option overcomes all issues listed above. However, as it carries numerous degrees of freedom, it also requires a strong prior or regularization.
We call this option the {\em piecewise linear} (PL) parametrization
\begin{equation}
  \label{eq:PL_zeta}
\zeta_{\rm PL}: \; \theta \mapsto \zeta_{\rm PL}(\theta), \quad \R^N \to H^1(\cup E_m),
\end{equation}
where $N$ is the number of finite element nodes on $\cup E_m$.
To assure nonnegativity of the boundary conductivity without having to introduce constraints in the optimization problem \eqref{eq:minimization}, the parameter vector $\theta \in \R^N$ defines the nodal values as $\zeta_1 = \theta_1^2, \dots, \zeta_N = \theta_N^2$. The boundary conductivity on the rest of the extended electrodes is then simply the linear interpolant of the neighboring nodal values. We tacitly assign the boundary nodes on $\partial \Omega \setminus \cup E_m$ with zero boundary conductivities, thus extending $\zeta_{\rm PL}$ to the whole boundary $\partial \Omega$ as an element of $H^1(\partial \Omega)$. Consult \cite{Hyvonen17b} for information on the effect of such a parametrization on the convergence of FEMs for the CEM.

The PL parametrization is regularized assuming a similar Gaussian smoothness prior as for the domain conductivity. To this end, let $\mathcal{J}_m$ denote the index set for the nodal indices on the closure of the $m$th extended electrode. A {\em preliminary} covariance matrix for the parameters defining the nodal values of the boundary conductivity as their squares is then formed via
\begin{equation}
  \label{eq:cov_zeta_PL}
	\big( \tilde{\Gamma}_{\theta} \big)_{ij} =
	\begin{cases}
		\gamma_\theta^2 \exp\left( -\frac{d_{\partial \Omega}(x_i, x_j)^2}{2 \lambda_\theta^2} \right) \qquad &{\rm if} \  i,j \in \mathcal{J}_m \ {\rm for} \ {\rm some} \ m \in \{1, \dots, M\},\\[2mm]
		0 \quad &\text{otherwise},
	\end{cases}
\end{equation}
where $d_{\partial \Omega}(\cdot, \cdot)$ measures the Euclidean distance along the boundary curve, $\gamma_\theta^2$ is the componentwise variance and $\lambda_\theta$ is the correlation length along the boundary. The expected value $\theta^\mu$ is simply set to zero, as in our numerical experiments its value did not significantly affect the reconstructions. The zero-mean Gaussian density defined by these parameters is finally conditioned by fixing the contact conductivities at the end points of the electrodes to zero; the corresponding covariance matrix $\Gamma_\theta$ for the parameters defining the nodal values on the interior of the extended electrodes can be obtained by forming a suitable Schur complement based on $\tilde{\Gamma}_\theta$ (see,~e.g.,~\cite{Kaipio05}).

Our second parametrization is a modified version of the ``hat-CEM'' introduced in~\cite{Hyvonen17b}. However, instead of letting the support of a hat-shaped boundary conductivity to cover a whole (extended) electrode, we allow its location and width to vary on each extended electrode. More precisely,  we define the {\em parametric hat} (PH) boundary conductivity  electrode-wise through
\begin{equation}
  \label{eq:PH_zeta}
  (\zeta_{{\rm PH}})_m :
  \left\{
  \begin{array}{l}
    (h_m, l_m, w_m) \mapsto (\zeta_{{\rm PH}})_m(h_m, l_m, w_m), \\[2mm]
    \big\{ (h,l,w) \in  \R_+^3 \, | \, l - \frac{w}{2} \geq 0 \ \text{and} \ l+\frac{w}{2} \leq 1\big\} =: \mathcal{Z}'_{\rm PH} \to  H^1(E_m),
  \end{array}
  \right.
\end{equation}
with
\begin{equation}
	\label{eq:phhats}
	(\zeta_{{\rm PH}})_m(h_m, l_m, w_m)(t) =
	\begin{cases}
		\frac{2 h_m (-2l_m + w_m + 2t)}{ w_m^2} \quad &\text{if} \quad t \in (l_m - \frac{w_m}{2}, l_m),\\[1mm]
		\frac{2 h_m (2l_m + w_m - 2t)}{w_m^2} \quad &\text{if} \quad t \in [l_m, l_m + \frac{w_m}{2}),\\[1mm]
		0 &\text{otherwise,}
	\end{cases}
\end{equation}
where $t$ takes values in the interval $(0, 1)$ as a normalized parametrization for the curve segment $E_m$. Observe that the parameters $l_m$, $w_m$ and $h_m$ have clear physical interpretations; they correspond, respectively, to the (normalized) center point of the electrode contact, the (normalized) width of the contact and the net contact conductance. However, as with the standard CEM, there is no reason to actually expect the boundary conductivity distribution to follow the shape of a hat function. The main motivation for the parametrization \eqref{eq:phhats} is that it is arguably the simplest `movable' $H^1(\partial \Omega)$ parametrization for the boundary conductivity.

Combining all extended electrodes and again setting the values of the boundary conductivity on $\partial \Omega \setminus \cup E_m$ to be zero, the parametrization \eqref{eq:phhats} naturally extends to the whole domain boundary as
\begin{equation*}
	\zeta_{\rm PH} : \theta = (\mathbf{h}, \mathbf{l}, \mathbf{w}) \mapsto \zeta_{\rm PH}(\mathbf{h}, \mathbf{l}, \mathbf{w}) = \zeta_{\rm PH}(\theta), \quad (\mathcal{Z}_{\rm PH}')^{M}  \to H^1(\partial \Omega).
\end{equation*}
As any reasonable finite element mesh has more than three edges per an (extended) electrode, the PH parametrization leads to a significantly lower number of degrees of freedom than the PL parametrization. For the sake of completeness, we have included the representation formulas for computing the derivatives of the electrode measurements with respect to $\mathbf{h}$, $\mathbf{l}$ and $\mathbf{w}$, needed for constructing the Jacobian matrix with respect to the parametrization of the boundary conductivity, in Appendix~\ref{app:PHderivatives}; cf.~Remark~\ref{rem:paramD}.

We found the standard Tikhonov regularization with a simple diagonal weight matrix adequate for regularizing the PH boundary conductivity parametrization. That is,
\begin{equation}
  \label{eq:cov_zeta_PH}
	\Gamma_{\theta} =
	\begin{bmatrix}
		\Gamma_h & & \\
		& \Gamma_l & \\
		& & \Gamma_w
	\end{bmatrix}
	=
	\begin{bmatrix}
	\gamma_h^2 {\rm I} & & \\
	& \gamma_l^2 {\rm I} & \\
	& & \gamma_w^2 {\rm I},
	\end{bmatrix} \in \R^{3M \times 3M},
\end{equation}
where $\gamma_h^2, \gamma_l^2, \gamma_w^2 > 0$ are the variances for the three parameter families.
Choosing the free parameters for the PH parametrization is considered in Section~\ref{sec:numerics}.

\begin{remark}
  \label{remark:PH}
  For the PH parametrization, we modify the line search step in Algorithm~\ref{alg:minimization} to ensure the contact conductivity parameters stay within $\mathcal{Z}'_{\rm PH}$ and thus define hat-shaped functions supported on the closures of the extended electrodes.  For a total parameter vector $\tau_j + t d_j$ under consideration, denote by $(h_{m,j},l_{m,j}, w_{m,j})$ the corresponding parameters defining the contact conductivity on the $m$th extended electrode according to~\eqref{eq:phhats}. Before evaluating the Tikhonov functional \eqref{eq:minimization} at $\tau_j + t d_j$, the parameter triplet $(h_{m,j},l_{m,j}, w_{m,j})$ is clamped inside $\mathcal{Z}'_{\rm PH}$ for each $m=1,\dots, M$ via the following steps (cf.~\eqref{eq:PH_zeta}): {\em (i)} If $w_{m,j} > 1$, then redefine $w_{m,j} = 1$. {\em (ii)} If $l_{m,j} - \frac{w_{m,j}}{2} < 0$, then redefine $l_{m,j} =  \frac{w_{m,j}}{2}$. {\em (iii)}  If $l_{m,j} + \frac{w_{m,j}}{2} > 1$, then redefine $l_{m,j} =  1 - \frac{w_{m,j}}{2}$. In particular, this modification means the partial convergence proof presented in the following section does not cover the case of the PH parametrization.
  \end{remark}

\subsection{On convergence of the algorithm}

For the PL parametrization of the boundary conductivity, a limited form of convergence of the Gauss–Newton method of \cref{alg:minimization} can be deduced from basic results \cite{Nocedal99}. The PH parametrization is more involved due to the clamping of the conductivity parameters to the admissible domain; see~Remark~\ref{remark:PH}. If a computationally expensive weighted projection was performed instead of simple clamping, and instead of performing a line search, local growth conditions were satisfied, the results of \cite{salzo2012convegence} could be used. Alternatively, if we were to add a proximal relaxation term to the method, the results of \cite{jauhiainen2019gaussnewton} could be used, which allow for such constraints and even nonsmooth regularization.

\begin{theorem}
	\label{thm:convergence}
	Assume that the PL parametrization \eqref{eq:PL_zeta} is employed for the boundary conductivity and suppose $\Gamma_\kappa, \Gamma_\theta \le \gamma \Id$ for some $\gamma>0$. Write $F$ for the objective of \eqref{eq:minimization}.
	Let $\{\tau_j\}_{j=0}^\infty$ be generated by \cref{alg:minimization} with the step lengths $t_j$ satisfying for given parameters $0 < c_1 < c_2 < 1$ the Wolfe conditions
        	\begin{align*}
		F(\tau_j+t_jd_j) & \le F(\tau_j) + c_1 t_j F'(\tau_j)d_j
		&& \text{(sufficient decrease)}\quad\text{and}
		\\
		F'(\tau_j+t_jd_j)d_j & \ge c_2 F'(\tau_j)d_j
		&& \text{(curvature)}.
	\end{align*}
	Then $F'(\tau_j) \to 0$, the set of accumulation points $\hat\tau$ of $\{\tau_j\}_{j \in \N}$ is non-empty, and all of them are critical, i.e., $F'(\hat\tau)=0$.
	Moreover, there always exist subintervals of $[0, \infty)$ on which $t_j$ satisfies the Wolfe conditions.
\end{theorem}

\begin{proof}
  It follows from Theorem~\ref{lemma:cont} and the solution formula \eqref{eq:zeta_sampling} that $U$ is continuously differentiable as a function of $(\sigma, \zeta)$ with $I$ fixed.
  The same is then true for the finite dimensional function $\mathcal{U}(\kappa, \theta)$ appearing in \eqref{eq:minimization} due to the smoothness of the functions $s \mapsto \exp(s)$ and $s \mapsto s^2$; see \eqref{eq:PL_zeta}.
	It follows that $F$ is continuously differentiable.
	That the Wolfe conditions can be satisfied now follows from \cite[Lemma 3.1]{Nocedal99} and the continuous differentiability of $F$.

        For any real parameter $r$, we denote $\lev_r F := \{ \tau \mid F(\tau) \le r\}$ the $r$-sublevel set of $F$.
        	Since $F$ is coercive due to the regularization terms in \eqref{eq:minimization} and the bound $\Gamma_\kappa, \Gamma_\theta \le \gamma \Id$, the sublevel set $\lev_{F(\tau_0)}$ is bounded.
	Due to the finite-dimensionality of the domain, and the continuous differentiability of $\mathcal{U}$, it follows using a compact covering argument that $J=(\mathcal{U}_\sigma, \mathcal{U}_\theta)$ is Lipschitz within $\mathcal{N}$ for any bounded open set $\mathcal{N} \supset \lev_{F(\tau_0)} F$.
	Due to the sufficient decrease in the Wolfe conditions as well as $d_j$ being a direction of descend for $F$, i.e., $F'(\tau_j)d_j \le 0$, see the proof of \cite[Theorem 10.1]{Nocedal99}, we also have the monotonicity $F(\tau_{j+1}) \le F(\tau_j)$. Consequently $\{\tau_j\}_{j \in \N} \subset \lev_{F(\tau_0)}$.

	To prove convergence, we set
	\[
	    R(\tau) := L\begin{pmatrix}\mathcal{U}(\tau)-\mathcal{V} \\ \tau-\tau^\mu \end{pmatrix}.
	\]
	Then $F(\tau) = \norm{R(\tau)}^2$ whereas the least squares problem \eqref{eq:gn} can be written more compactly as $\min_{\Delta \tau} \norm{R(\tau)+R'(\tau)\Delta \tau}^2$, where
	\[
	    R'(\tau)=L\begin{pmatrix}J(\tau) \\ \Id\end{pmatrix}.
	\]
	That $F'(\tau_j) \to 0$ will follow from \cite[Theorem 10.1]{Nocedal99} after we verify that
\begin{itemize}
\item[a)]$R'$ is Lipschitz in an open bounded neighborhood $\mathcal{N}$ of $\lev_{F(\tau_0)} F$,
\item[b)] $R'(\tau)^\top R'(\tau) \ge  \tilde\gamma \Id$ for some $\tilde\gamma > 0$.
\end{itemize}
	As $L^\top L = \diag(\Gamma^{-1}_{\rm noise}, \Gamma^{-1}_\kappa, \Gamma^{-1}_\theta)$, we have $L^\top L \ge \begin{bsmallmatrix} \Gamma_{\text{noise}}^{-1} & 0 \\ 0 & \gamma^{-1} \Id \end{bsmallmatrix}$.
	It follows that $R'(\tau)^\top R'(\tau) =   J(\tau)^\top \Gamma_{\text{noise}}^{-1} J(\tau) + \gamma^{-1} \Id$, so that b) holds with $\tilde\gamma=\gamma^{-1}$.
	On the other hand, condition a) holds if $J$ is Lipschitz within $\mathcal{N}$, which we have also shown.
	Thus $F'(\tau_j) \to 0$.

	The existence of a subsequence $\{j_k\}_{k \in \N} \subset \N$ such that $\tau_{j_k}$ converges to some $\hat\tau$ follows from finite-dimensionality and the boundedness of $\lev_{F(\tau_0)} F$.
	Since, for any such subsequence, $F'(\tau_{j_k}) \to 0$ and $F'$ is continuous by the continuity and continuous differentiability of $\mathcal{U}$, necessarily $F'(\hat\tau)=0$.
\end{proof}

\section{Numerical examples}
\label{sec:numerics}

In this section we present some experimental results that support the use of our new model in connection with real world EIT measurements. To put it short, including the estimation of a boundary contact function as a part of a reconstruction algorithm for EIT can result in qualitatively better reconstruction of the primary unknown, i.e.~the domain conductivity, if it is expected that the electrode positions may have been mismodeled. Similar conclusions have previously been reached in~\cite{Candiani19,Darde12,Darde13a,Darde13b,Hyvonen17c} by resorting to shape derivatives of EIT measurements with respect to the electrode locations. Having (slightly) incomplete information on the electrode positions is an important case in practice as the contact locations are exactly known only in trivial experimental settings such as water tanks. Be that as it may, we resort to water tank experiments in what follows.

We consider the same data as in~\cite{Hyvonen17b}\footnote{Electrical impedance dataset of thorax shaped tank \href{https://doi.org/10.5281/zenodo.4475417}{doi:10.5281/zenodo.4475417}}. The measurements were performed using the \emph{Kuopio impedance tomography} (KIT4) device \cite{Kourunen09} on a thorax-shaped tank having a circumference of 106\,cm; see Figure~\ref{fig:nontrivialreconstruction}(b). The tank was equipped with 16 near-equidistant 2\,cm wide rectangular electrodes that reached from the bottom of the tank to the water surface at the height of 5\,cm. The measurements were conducted using low-frequency common ground current patterns with amplitudes of approximately 1.2\,mA. The phase information in the potential measurements is ignored and their amplitudes are treated as if they had resulted from measurements with direct current. Before forming any reconstructions, the current amplitudes (and the respective electrode potentials) are scaled to 1.0\,mA to avoid uneven weighting between current patterns. As the measurement setup is homogeneous in the vertical direction, it can be modeled by a two-dimensional version of the forward problem~\eqref{eq:cemeqs}. For further details on the measurement setup, see \cite{Hyvonen17b}. The source code of our implementation has been made available on Zenodo at \href{https://doi.org/10.5281/zenodo.4475524}{doi:10.5281/zenodo.4475524}

In all experiments, the two-dimensional computational domain is triangulated as a FE mesh consisting of 2725 nodes, being densest near the (extended) electrodes and sparsest near the center of the domain. The triangulations are generated using a slightly modified version of Triangle~\cite{Shewchuk96b}\footnote{Triangle with disable attribute interpolation switch by Kuutela 2020 \href{https://doi.org/10.5281/zenodo.4472025}{doi:10.5281/zenodo.4472025}}.

In order to provide a reasonably fair comparison between different approaches, we use the same values for the constants appearing in the prior covariances \eqref{eq:cov_kappa}, \eqref{eq:cov_zeta_PL} and \eqref{eq:cov_zeta_PH} in all tests. The choices are based on a grid search and manual verifications across several test cases. The employed parameter values are listed in Table~\ref{tab:regvalues}. Take note that we have made the naive choice of $\Gamma_{\rm noise} = \mathrm{I} \in \R^{240}$ in our numerical tests, which means that the standard deviations in Table~\ref{tab:regvalues} are in fact the ratios between the standard deviations for the listed parameters and that for the additive zero-mean noise that is assumed to have the same variance and be mutually independent over all potential measurements. Moreover, the weighted norm $\| \cdot \|_{\Gamma_{\rm noise}^{-1}}$ appearing in \eqref{eq:minimization} simply becomes the Euclidean norm $\|  \cdot \|_2$.  When computing reconstructions based on the traditional CEM with constant contacts (over the actual electrodes), no regularization with respect to the contact conductances is needed or used.

We consider three general geometries for the interplay between the true and extended electrodes. First, we present results for (almost) exactly known electrode positions, that is, the extended electrodes $E_m$ coincide with $e_m$ for $m=1, \dots, M$ up to the available information on the precise positions of the latter on the interior surface of the water tank. In particular, the available information on the electrode positions is used to its full extent when forming the reconstruction with all three parametrization,~i.e.~the standard CEM as well as the PL and PH parametrizations for the boundary conductivity in the new electrodeless model. In the second and third cases, the extended electrodes are approximately 12\,mm and 22\,mm wider, respectively, than the physical electrodes, with the corresponding extensions divided (uniformly) randomly between the two endpoints of the electrodes. When computing reconstructions based on the standard CEM in these cases, we assume that the electrode widths are known but intentionally mismodel the setup by placing each computational electrode in the middle of its extended counterpart. This amounts on average to approximately 3\,mm and 6\,mm and at most 6\,mm and 11\,mm of displacement, respectively, for the computational electrodes compared to the true physical setup when the standard CEM is used for forming the reconstructions.

\begin{table}[t]
	\begin{tabular}{lll}
		Variable & Value & Description \\
		\hline
		$\gamma_\kappa$ & $10$\,V$^{-1}$ & standard deviation for the log-conductivity\\
		$\lambda_\kappa$ & $3 \cdot 10^{-2}$\,m & correlation length for the log-conductivity\\
		$\gamma_\theta$ & $5 \cdot 10^2$\,S/V  & standard deviation in the PL parametrization\\
		$\lambda_\theta$ & $3 \cdot 10^{-3}$\,m & correlation length for the PL parametrization\\
		$\gamma_h$ & $10^3$\,S/(m\,V) & standard deviation for $h$ in the PH parametrization\\
		$\gamma_l$ & $10^{1.5}$\,V$^{-1}$ & standard deviation for $l$ in the PH parametrization\\
		$\gamma_w$ & $10^2$\,V$^{-1}$ & standard deviation for $w$ in the PH parametrization
	\end{tabular}
	\caption{The parameters employed in \eqref{eq:cov_kappa}, \eqref{eq:cov_zeta_PL} and \eqref{eq:cov_zeta_PH} throughout the numerical experiments.}
	\label{tab:regvalues}
\end{table}

\subsection{Comparison with experimental data}
\label{sec:inclusionlessdom}
In our first experiment, we consider a ``trivial'' case where the tank is filled with mere tap water without any embedded inhomogeneities. The domain conductivity is parametrized by a single variable with no regularization,~i.e.~no penalty term for the conductivity is included in the minimized Tikhonov functional. We are particularly interested in the final residual
\begin{equation}
  \label{eq:residual}
\| \mathcal{U}(\kappa_{\rm opt}, \theta_{\rm opt}) - \mathcal{V} \|_2
\end{equation}
and the resulting optimal constant conductivity level
$\sigma_{\rm opt} = \exp \kappa_{\rm opt}$ for the considered models,~i.e.~the standard CEM and the electrodeless model with the PL and PH parametrizations for the boundary conductivity. Our secondary interest lies with the logarithms of the estimated net electrode conductances $\log \int_{E_m} \zeta(\theta_{\rm opt}) \, {\rm d}s$, $m=1, \dots, M$, where $\theta_{\rm opt}$ denotes the optimal parameters for the boundary conductivity produced by Algorithm~\ref{alg:minimization} for the considered model and $\zeta = \zeta_{\rm CEM}$, $\zeta = \zeta_{\rm PL}$ or $\zeta = \zeta_{\rm PH}$.

As mentioned in Section~\ref{sec:zetaoptions}, the expected value for the contact conductance in the PL parametrization is the zero function, whereas those for the parameters appearing in the PH parametrization are zero for the net conductance $h$, the center of the extended electrode for the center of the hat function $l$, and the relative length of the true electrode compared to the employed extended electrode for the hat width $w$. The initial guesses for the contact conductivities in the PL and PH parametrizations, as well as for the standard CEM, are chosen so that the net conductance over each electrode is 0.001\,S. For the PL model, the initial nodal values are the same over the whole extended electrode, whereas for the PH model, the initial hat is at the center of the corresponding extended electrode, and it has the same width as the physical electrodes. In each case, Algorithm~\ref{alg:minimization} is run until obvious convergence, or up to 50 iterations.

The residuals and the estimated domain conductivity levels for the three models and the three geometric settings for the interplay between the electrodes and the extended electrodes are listed in Table~\ref{tab:trivialresults}, whereas the logarithms of the estimated net electrode conductances and their variances over the electrodes are given in Table~\ref{tab:trivialcontacts}. For comparison, the conductivity levels estimated for the same measurement setup in \cite{Hyvonen17b} were 0.022722\,S/m for the standard CEM and 0.022720\,S/m for a hat shaped contact conductivity model with known electrode positions and no regularization with respect to the height of the hat in the inversion algorithm. The corresponding reconstructed contact conductance levels at the electrodes are shown in \cite[Fig.~7]{Hyvonen17b}.

\begin{table}[t]
	\begin{center}
	    \begin{tabular}{l|rrr|rrr}
	      \multicolumn{1}{c}{} &
			\multicolumn{3}{c}{Residuals}& \multicolumn{3}{c}{Domain conductivities} \\
			& Exact & 12\,mm & 22\,mm & Exact & 12\,mm & 22\,mm \\
			\hline
			CEM & 0.0795 & 0.150 & 0.301 & 0.0228 & 0.0229 & 0.0231 \\
			PL & 0.0450 & 0.0349 & 0.0971 & 0.0228 & 0.0227 & 0.0228 \\
			PH & 0.249 & 0.0382 & 0.0369 & 0.0227 & 0.0227 & 0.0227
		\end{tabular}
	\caption{Residuals \eqref{eq:residual} and the estimated conductivity levels (S/m) for a tank without embedded inhomogeneities.}
	\label{tab:trivialresults}
	\end{center}
\end{table}

When the electrode positions are modeled (almost) exactly, the PH parametrization results in a larger residual \eqref{eq:residual} than the standard CEM and the PL model. The exact reason for this discrepancy is unknown, but it is presumably related to the prior model associated to the PH parametrization. All estimated domain conductivities are well aligned with the values from~\cite{Hyvonen17b}.

When the extended electrodes are 12\,mm wider than the true ones,~i.e.~the computational electrodes for the standard CEM are misplaced on average by 3.0\,mm, the residuals corresponding to the PL and, especially, the PH parametrization decrease whereas the residual for the standard CEM increases.
All estimates for the domain conductivity are still within one per cent of the values reported in~\cite{Hyvonen17b}.
It seems that the slight widening of the extended electrodes gives the PL and PH parametrizations more freedom to match the measurement data, whereas the mismodeling of the true electrodes slightly impairs the performance of the standard CEM in this regard.

Increasing the difference in widths between the extended and true electrodes to 22\,mm, and thus the average misplacement of the computational CEM electrodes to 5.6\,mm, essentially highlights the observations regarding the residuals \eqref{eq:residual} in the previous case: employing the standard CEM leads to a residual that is significantly larger than those for the PL and PH parametrizations. Both PH and PL models lead to domain conductivity estimates that are still in the vicinity of those listed in \cite{Hyvonen17b}, while the estimate produced by the CEM is somewhat inaccurate with the algorithm also having some trouble converging.

The means and standard deviations for the logarithms of the reconstructed net electrode conductances are listed in Table~\ref{tab:trivialcontacts}. The net conductances are on average largest for the PH parametrization, which is once again in line with the observations in~\cite{Hyvonen17b}. In particular, the net integrals of the contact conductivity over the electrodes is obviously not the only property of the contacts that affects the (simulated) electrode measurements, but the shape of the contact conductivity over individual electrodes also plays a role.
Mismodeling of the electrode positions seems to decrease the estimated net conductances for the standard CEM and the PL parametrization.
One of the electrodes, i.e.~number 15, seems to have a significantly better contact than the others, which does not have any clear explanation, but the observation is anyways aligned with the material in~\cite{Hyvonen17b}.

\begin{table}[t]
  \begin{center}
    \begin{tabular}{l|rrr|rrr}
      \multicolumn{1}{c}{} &
	  \multicolumn{3}{c}{Means for $\log \int_{E_m} \zeta \, {\rm d}s$}& \multicolumn{3}{c}{Standard deviations} \\
		& Exact & 12\,mm & 22\,mm & Exact & 12\,mm & 22\,mm \\
		\hline
		CEM & -3.65 & -3.82 & -4.19 & 0.481 & 0.363 & 0.246 \\
		PL & -2.61 & -4.47 & -4.83 & 0.518 & 0.402 & 0.312 \\
		PH & -3.39 & -3.34 & -3.40 & 0.179 & 0.272 & 0.335
	\end{tabular}
	\caption{Means and standard deviations for the logarithms of the reconstructed net electrode conductances (S) for a tank without embedded inhomogeneities.}
	\label{tab:trivialcontacts}
        \end{center}
\end{table}

\subsection{Reconstructions from experimental data}
\label{sec:inclusiondom}
We next consider reconstructing a nontrivial conductivity distribution inside the same water tank. Two inclusions are placed in the tank: an insulating plastic cylinder and a rectangular metal pipe as shown in Figure~\ref{fig:nontrivialreconstruction}(b). They are both homogeneous in the vertical direction and break the water surface, making the measurement configuration once again modelable by a two-dimensional version of \eqref{eq:cemeqs}. This time the conductivity field is discretized using the same 2725 piecewise linear FE basis functions as the ones used for solving \eqref{eq:cemeqs} in the two-dimensional domain. Encouraged by the experiments with the empty water tank, we choose $\sigma_0 = \sigma^\mu = 0.02$\,S/m as both the homogeneous initial guess and the expected value for the domain conductivity. The initial guesses and expected values for the PL and PH contact conductivity parametrizations are as in the previous experiment.  In each reconstruction process, we let Algorithm~\ref{alg:minimization} run until clear convergence or up to 50 iterations. The models for the considered extended electrodes and the resulting errors in the placement of the computational electrodes for the standard CEM are the same as for the experiments with the empty tank.

\begin{figure}[t]
	\centering
	\subfloat[Reconstructions]{\includegraphics[width=0.8\textwidth]{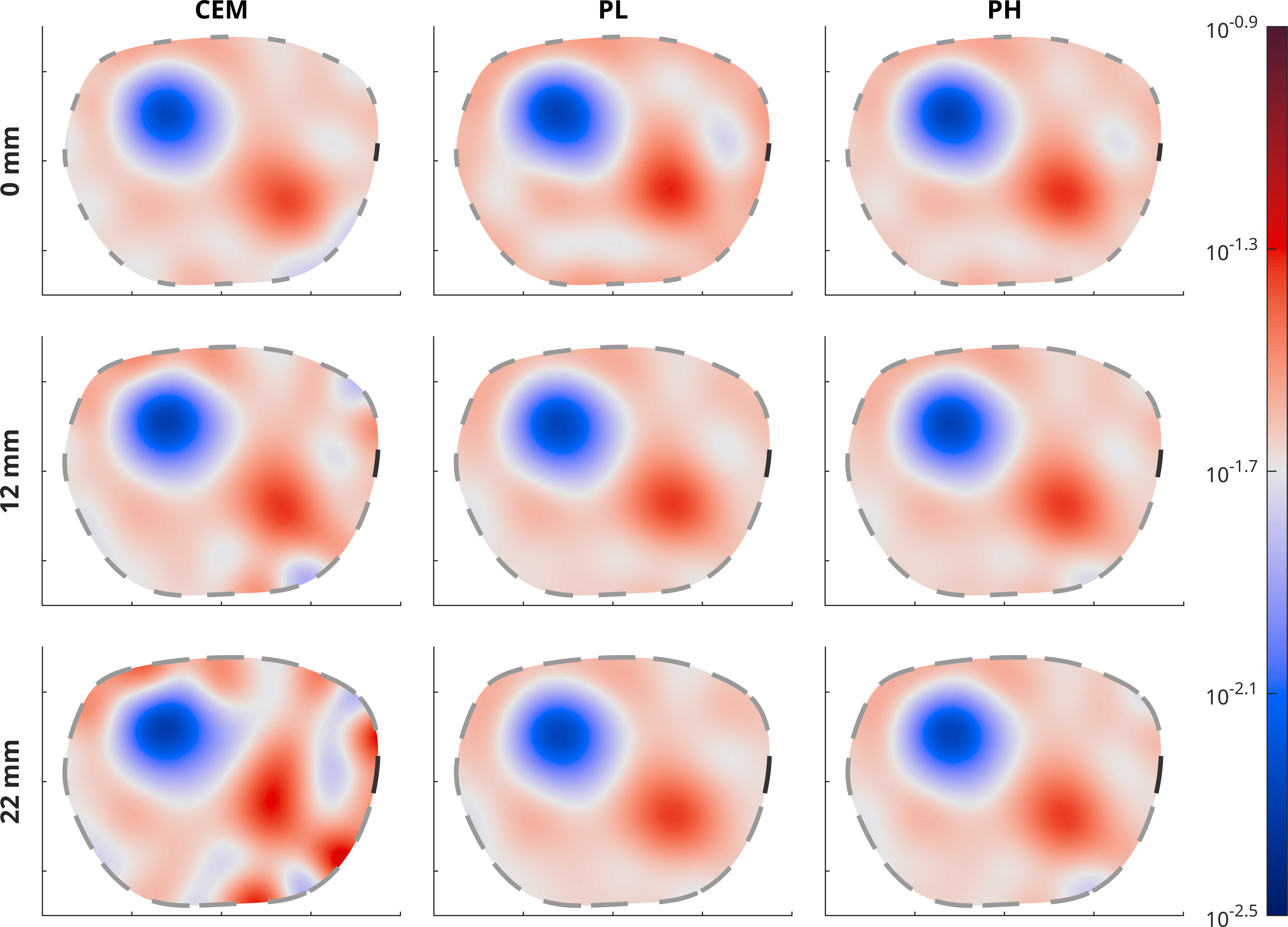}}\\
	\subfloat[Target]{\includegraphics[angle=180,width=0.33\textwidth]{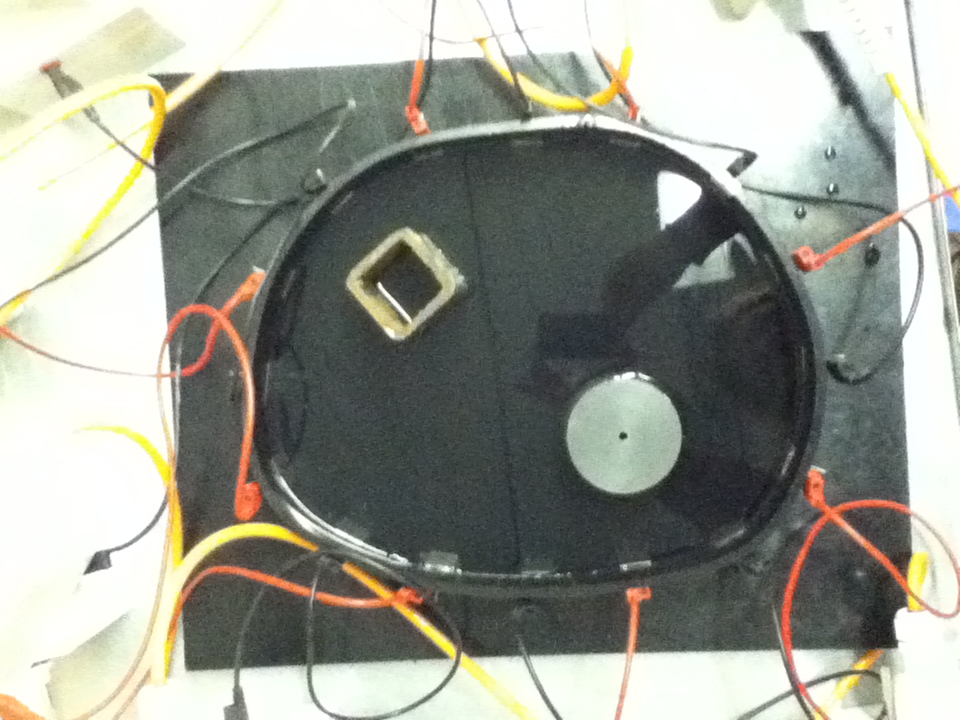}}
	\caption{Domain conductivity reconstructions (S/m) for a water tank with an insulating (top left) and a well conducting (bottom right) inclusions. The columns correspond to the standard CEM and the new electrodeless model with the PL and PH parametrizations. The rows are associated with different levels of error in the model for the electrode positions. The darkest extended electrode corresponds to $m=1$ and the others are numbered in the positive directions.}
	\label{fig:nontrivialreconstruction}
\end{figure}

The reconstructions for the three forward models, as well as for the three configurations for the (computational) electrodes and extended electrodes, are shown in Figure~\ref{fig:nontrivialreconstruction}. An immediate observation is that the quality of the reconstructions produced by the PL and PH parametrizations for the domain conductivity does not suffer noticeably from the mismodeling of the electrode positions; the freedom in these models for fine-tuning the contact conductivity shapes along the extended electrodes seems to provide enough flexibility for coping with the geometric modeling error. On the other hand, the reconstructions corresponding to the standard CEM significantly deteriorate as the mean amount of misplacement in the (computational) electrodes increases from 0 to 5.6\,mm. To be more precise, conductivity artifacts start to appear especially near the domain boundary.

Table~\ref{tab:mincomponents} shows the final values for the different quadratic terms appearing in the minimized Tikhonov functional \eqref{eq:minimization} for the three forward models and the extended electrodes that are 22\,mm wider than the true ones, corresponding on average to a misplacement of 5.6\,mm for the computational electrodes in the standard CEM. (Recall that there is no penalization with respect to the constant contact conductances in  \eqref{eq:minimization} for the standard CEM, which explains the missing value in Table~\ref{tab:mincomponents}.) It is obvious that with the considered prior model for the domain log-conductivity, the standard CEM with the erroneous electrode positions is not able to bring the data fit term $\| \mathcal{U}(\kappa, \theta) - \mathcal{V} \|_2$ down to the same level as the more flexible PL and PH parametrizations. A lower value for this residual could be achieved by relaxing the prior for $\kappa = \log\sigma$, but this would inevitably lead to even more severe artifacts in the associated reconstruction of $\sigma$.

\begin{table}[t]
	\centering
	\begin{tabular}{l|ccc}
		& $\| \mathcal{U}(\kappa, \theta) - \mathcal{V} \|_2$ & $\| \kappa - \kappa^\mu \|_{\Gamma^{-1}_\kappa}$ & $\| \theta - \theta^\mu \|_{\Gamma^{-1}_\theta}$\\
			\hline
			CEM & 0.255 & 0.0318 & --- \\
			PL & 0.0182 & 0.0187 & 0.00864 \\
			PH & 0.0186 & 0.0195 & 0.0128
	\end{tabular}
	\caption{Final values of the three terms in \eqref{eq:minimization} for the standard CEM and the PL and PH parametrizations for the boundary conductivity. There is no penalization with respect to the constant contact conductivities in the standard CEM.}
	\label{tab:mincomponents}
\end{table}

\subsection{Contact conductance reconstructions}
\
A natural question is whether our proposed methods actually employ the flexibility in the boundary conductivity parametrizations to reconstruct the electrode positions or do the reconstructions merely fit to the measurements with no clear physical interpretation. To this end, Figure~\ref{fig:contacts} visualizes the reconstructed boundary conductivities for the experiments with the empty water tank in Section~\ref{sec:inclusionlessdom} (left column) and with the tank enclosing the two cylindrical inclusions in Section~\ref{sec:inclusiondom} (right column). All images in Figure~\ref{fig:contacts} correspond to extended electrodes that are 22\,mm wider than the physical ones. According to a qualitative assessment of the rows in Figure~\ref{fig:contacts}, the reconstructions of the contact conductivity associated to the PH parametrization are similar for the empty tank and the variable conductivity case on each electrode, although the equivalence cannot be described as perfect. A similar conclusion does not hold on all electrodes for the PL parametrization; for the empty tank, the contact is divided on some extended electrodes into two humps located near the end points, which must be considered a bad approximation of the underlying physical reality.

\begin{figure}[t]
	\centering
	\includegraphics[width=0.85\textwidth]{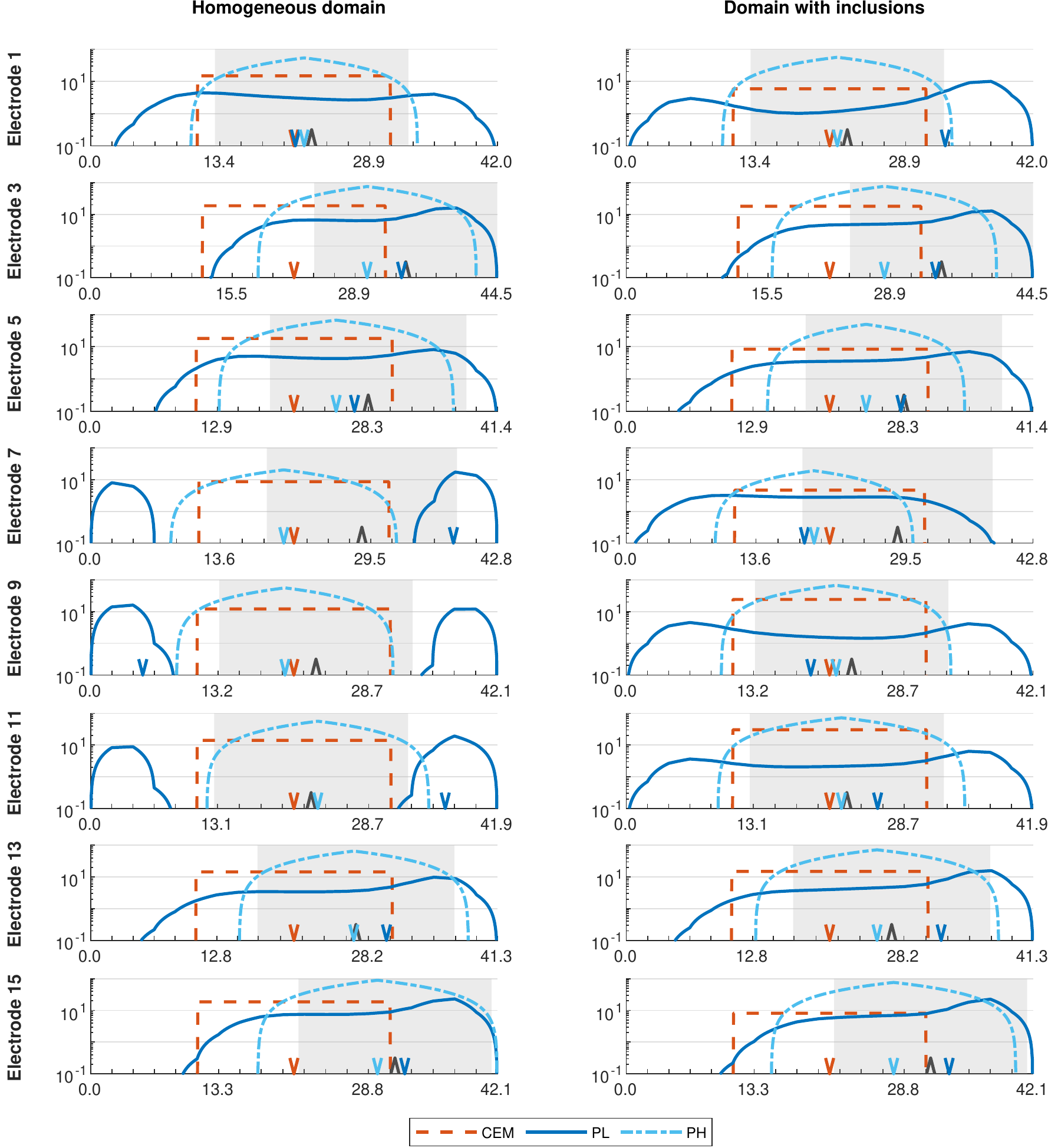}
 	\caption{Contact conductivity reconstructions on the odd electrodes for the experiments in Section~\ref{sec:inclusionlessdom} (left) and Section~\ref{sec:inclusiondom} (right) with 22\,mm electrode extensions. The horizontal axis represents arclength parametrization in the counter-clockwise direction and in millimeters for the extended electrode, with the ticks indicating nodes in the FE mesh. The vertical axis is logarithmic in S/m$^2$. The physical electrode position is indicated by a gray background and the red line depicts the reconstructed piecewise constant contact conductivity for the standard CEM at the (randomly) erroneous position. The dark and light blue curves are the reconstructed contact conductivities for the PL and PH parametrizations, respectively. Each inverted triangle marks the center of the contact conductivity mass for a particular reconstruction, while the black upright triangle is the center of the true electrode.}
	\label{fig:contacts}
\end{figure}

In particular, according to Figure~\ref{fig:contacts}, Algorithm~\ref{alg:minimization} seems to usually move the center of mass for the PH contact conductivity parametrization toward the center of the actual physical electrode. To quantify such a claim, let $s_m \in \R$ denote the midpoint of the $m$th physical electrode in an arclength parametrization of $\partial \Omega$, and let $\hat{s}_m \in \R$ be the center of mass for some reconstructed boundary conductivity on $E_m$ with respect to that same arclength parametrization. We define the corresponding mean error in the localization of the electrode contacts as
\begin{equation}
	\label{eq:centererror}
{\rm err}_E =  \frac{1}{M} \sum_{m=1}^M \left| s_m - \hat{s}_m \right|.
\end{equation}
Table~\ref{tab:contacterrs} lists these mean errors in the reconstructed electrode positions for the experiments with 22\,mm electrode extensions in Sections~\ref{sec:inclusionlessdom} and \ref{sec:inclusiondom}, with ${\rm err}_E = 5.57$\,mm for the standard CEM corresponding to the random variation in the positions of the (extended) electrodes in the computational model. The numbers in Table~\ref{tab:contacterrs} confirm that the contact centers indeed move closer to the midpoints of the physical target electrodes for the PH parametrization, but for the PL parametrization the results do not indicate a capability to locate the true electrodes. In particular, it seems that the unphysical division of some contacts into two parts leads to a worse performance than the trivial guess (cf.~the CEM) for the PL parametrization in the case of the empty tank.

The general features of the numerical results documented in this (and the previous) sections were consistently reproduced with different randomizations of the electrode extensions and for almost all experimental data for the considered water tank at our disposal. That is, when Algorithm~\ref{alg:minimization} is applied to the PL and PH parametrizations, the reconstructed boundary conductivities tend to compensate for the mismodeled electrode positions regardless of the domain conductance. However, it is unclear whether this extends to more complicated geometric settings and/or for other regularization or statistical inversion strategies for tackling the inverse problem of practical EIT. In particular, it is unclear what would be the most reasonable prior model for the PL parametrization of the contact conductivity. The prior used in this work allows unphysical shapes for the reconstructed PL contacts.

\begin{table}[t]
	\centering
	\begin{tabular}{l|cc}
		& Empty tank & With inclusions \\
		\hline
		CEM & 5.57 & 5.57 \\
		PL & 7.29 & 4.92 \\
		PH & 2.79 & 3.31
	\end{tabular}
	\caption{Reconstructed errors in the electrode contact midpoints \eqref{eq:centererror} for the experiments in Sections~\ref{sec:inclusionlessdom} and \ref{sec:inclusiondom} with 22\,mm electrode extensions.}
	\label{tab:contacterrs}
\end{table}

\section{Concluding remarks}
\label{sec:conclusions}

This work introduced a new approach to modeling electrodes  via introducing a spatially varying boundary admittivity function in the framework of the CEM of EIT under only generic information on the electrode locations. Two parametrizations for the boundary admittivity were considered: one directly employing the underlying piecewise linear FE basis and another corresponding to a single hat-shaped admittivity with varying width, height and location on each extended electrode. Our numerical experiments based on experimental water tank data demonstrated that reconstructing the boundary admittivity function as a part of an iterative Bayesian output least squares algorithm leads to far better reconstructions of the domain admittivity than simply ignoring the incomplete information on the electrode positions.

\appendix
\section{Derivatives for the PH parametrization}
\label{app:PHderivatives}
The derivatives of electrode measurements with respect to the free parameters in the PH electrode conductance parametrization \eqref{eq:phhats} can be straightforwardly deduced using Remark~\ref{rem:paramD}:
\begin{equation*}
	\begin{split}
		\frac{\partial_{h_m} U(\sigma, \zeta_{\rm PH}(\mathbf{h}, \mathbf{l}, \mathbf{w}), I)}{|E_m|} \cdot \tilde{I} =&
		-\int_{l_m - w_m/2}^{l_m} \frac{2 (-2 l_m + w_m + 2 t)}{w_m^2} \, (U - u) (\tilde{U} - \tilde{u}) \, {\rm d} t \\
		&-\int_{l_m}^{l_m + w_m/2} \frac{2 (2 l_m + w_m - 2 t)}{w_m^2} \, (U - u) (\tilde{U} - \tilde{u}) \, {\rm d} t,
                \\[2mm]
		\frac{\partial_{l_m} U(\sigma, \zeta_{\rm PH}(\mathbf{h}, \mathbf{l}, \mathbf{w}), I)}{|E_m|} \cdot \tilde{I} =&
		\int_{l_m - w_m/2}^{l_m} \frac{4h_m}{w_m^2} \, (U - u) (\tilde{U} - \tilde{u}) \, {\rm d} t \\
		&-\int_{l_m}^{l_m + w_m/2} \frac{4h_m}{w_m^2} \, (U - u) (\tilde{U} - \tilde{u}) \, {\rm d} t, \\[2mm]
		\frac{\partial_{w_m} U(\sigma, \zeta_{\rm PH}(\mathbf{h}, \mathbf{l}, \mathbf{w}), I)}{|E_m|} \cdot \tilde{I} =&
		\int_{l_m - w_m/2}^{l_m} \frac{2 h_m(-4l_m + w + 4t)}{ w_m^3} \, (U - u) (\tilde{U} - \tilde{u}) \, {\rm d} t \\
		&+\int_{l_m}^{l_m + w_m/2} \frac{2 h_m(4l_m + w - 4t)}{w_m^3} \, (U - u) (\tilde{U} - \tilde{u}) \, {\rm d} t,
	\end{split}
\end{equation*}
where the integrals correspond to a normalized parametrization of the $m$th extended electrode curve segment.

\bibliographystyle{siam}
\bibliography{noelec-refs}

\end{document}